\documentclass[12pt]{amsart}
\usepackage{graphicx}
\usepackage{amssymb}
\usepackage{amscd}
\usepackage{cite}
\usepackage[all]{xy}
\usepackage{url}
\newtheorem{tw}{Theorem}[section]
\newtheorem{prop}[tw]{Proposition}

\newtheorem{wn}[tw]{Corollary}

\theoremstyle{remark}
\newtheorem{uw}[tw]{Remark}

\theoremstyle{definition}
\newtheorem*{df}{Definition}
\newcommand{\cal}[1]{\mathcal{#1}}

\newcommand{\podz}{\subseteq}

\newcommand{\sig}{\sigma}
\newcommand{\eps}{\varepsilon}
\newcommand{\ro}{\varrho}

\newcommand{\kre}[1]{\overline{#1}}
\newcommand{\gen}[1]{\langle #1 \rangle}
\newcommand{\map}[3]{#1\colon #2\to #3}
\newcommand{\field}[1]{\mathbb{#1}}
\newcommand{\zz}{\field{Z}}
\newcommand{\kk}{\field{K}}
\newcommand{\rr}{\field{R}}

\newcommand{\st}{\;|\;}
\newcommand{\cM}{{\cal M}}

\newcommand{\lst}[2]{{#1}_1,\dotsc,{#1}_{#2}}

\begin{document}

\numberwithin{equation}{section}

\title[A finite presentation for the hyperelliptic\ldots]
{A finite presentation for the hyperelliptic mapping class group of a nonorientable surface}

\author{Micha\l\ Stukow}

\thanks{Supported by MNiSW N201 366436 and NCN 2012/05/B/ST1/02171.}
\address[]{
Institute of Mathematics, University of Gda\'nsk, Wita Stwosza 57, 80-952 Gda\'nsk, Poland }

\email{trojkat@mat.ug.edu.pl}


\keywords{Mapping class group, Nonorientable surface, Braid group, Hyperelliptic surface} \subjclass[2000]{Primary 57N05;
Secondary 20F38, 57M99}

\begin{abstract}
We obtain a simple presentation of the hyperelliptic mapping class group ${\cal M^h}(N)$ of a nonorientable surface $N$. As
an application we compute the first homology group of ${\cal M^h}(N)$ with coefficients in $H_1(N;\zz)$.
\end{abstract}

\maketitle%
\section{Introduction}%
Let $N_{g,s}^n$ be a smooth, nonorientable, compact surface of genus $g$ with $s$ boundary components and $n$ punctures. If
$s$ and/or $n$ is zero, then we omit it from the notation. If we do not want to emphasise the numbers $g,s,n$, we simply write
$N$ for a surface $N_{g,s}^n$. Recall that $N_{g}$ is a connected sum of $g$ projective planes, and $N_{g,s}^n$ is
obtained from $N_g$ by removing $s$ open disks and specifying the set $\Sigma=\{\lst{z}{n}\}$ of $n$ distinguished points in
the interior of $N_g$.

Let ${\textrm{Diff}}(N)$ be the group of all diffeomorphisms $\map{h}{N}{N}$ such that $h$ is the identity on each
boundary component and $h(\Sigma)=\Sigma$. By ${\cal{M}}(N)$ we denote the quotient group of ${\textrm{Diff}}(N)$ by
the subgroup consisting of maps isotopic to the identity, where we assume that isotopies fix $\Sigma$ and are the identity
on each boundary component. ${\cal{M}}(N)$ is called the \emph{mapping class group} of $N$. 

The mapping class group ${\cal{M}}(S_{g,s}^n)$ of an orientable surface is defined analogously, but we consider only
orientation preserving maps. If we include orientation reversing maps, we obtain the so-called \emph{extended mapping class
group} ${\cal{M}}^{\pm}(S_{g,s}^n)$.

Suppose that the closed orientable surface $S_g$ is embedded in $\rr^3$ as shown in Figure \ref{r01}, in such a way that it
is invariant under reflections across $xy,yz,xz$ planes. Let $\map{\ro}{S_g}{S_g}$ be the \emph{hyperelliptic involution},
i.e. the half turn about the $y$-axis.
\begin{figure}[h]
\begin{center}
\includegraphics[width=0.9\textwidth]{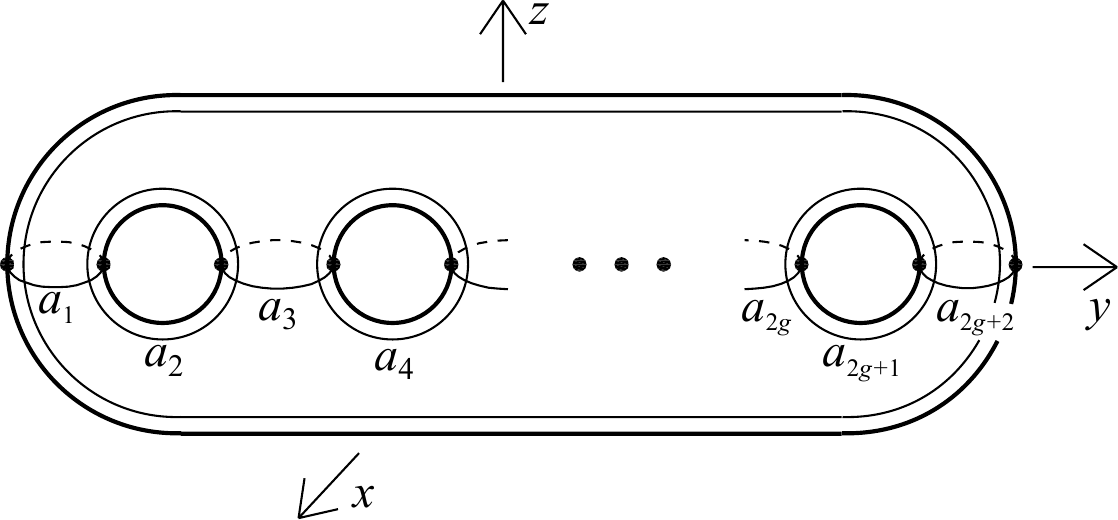}
\caption{Surface $S_g$ embedded in $\rr^3$.}\label{r01} %
\end{center}
\end{figure}
The \emph{hyperelliptic mapping class group} $\cM^h(S_g)$ is defined to be the centraliser of $\ro$ in $\cM(S_g)$.
In a similar way we define the \emph{extended hyperelliptic mapping class group} $\cM^{h\pm}(S_g)$ to be the
centraliser of $\ro$ in $\cM^{\pm}(S_g)$.
\subsection{Background}
The hyperelliptic mapping class group turns out to be a very interesting and important subgroup of the mapping class group.
Its algebraic properties have been studied extensively -- see \cite{Bir-Hil,Kawazumi1} and references there. Although
$\cM^h(S_g)$ is an infinite index subgroup of $\cM(S_g)$ for $g\geqslant 3$, it plays surprisingly important role in
studying its algebraic properties. For example Wajnryb's simple presentation \cite{Wajn_pre} of the mapping class group
$\cM(S_g)$ differs from the presentation of the group $\cM^h(S_g)$ by adding one generator and a few relations.
Another important phenomenon is the fact, that every finite cyclic subgroup
of maximal order in $\cM(S_g)$ is conjugate to a subgroup of $\cM^h(S_g)$ \cite{MaxHyp}.

Homological computations play a prominent role in the theory of mapping class groups. Let us mention that in the 
case of the hyperelliptic mapping class group, B\"odigheimer, Cohen and Peim \cite{BoedCoh} computed $H^*(\cM^h(S_g);\kk)$ with coefficients
in any field $\kk$. Kawazumi showed in \cite{Kawazumi1} that if $\textrm{ch}(\kk)\neq 2$ then 
$H^*(\cM^h(S_g);H^1(S_g;\kk))=0$. For the integral coefficients, Tanaka \cite{Tanaka} showed that 
 $H_1(\cM^h(S_g);H_1(S_g;\zz))\cong\zz_2$. Let us also mention that Morita 
\cite{MoritaJacFou} showed that in the case of the full mapping class group, $H_1(\cM(S_g);H_1(S_g,\zz))\cong\zz_{2g-2}$.
\subsection{Main results}
The purpose of this paper is to extend the notion of the hyperelliptic mapping class group to 
the nonorientable case. We define this group $\cM^h(N)$ in Section \ref{sec:def:hiper} and observe that it contains a 
natural subgroup $\cM^{h+}(N)$ of index 2 (Remark \ref{rem:positiv:subg}).

Then we obtain simple presentations of these groups (Theorems \ref{th:004} and \ref{tw:pres:pos}). By analogy with the orientable case, these presentations may be thought of as the first approximation of a presentation of the full mapping class group $\cM(N)$. In fact, for $g=3$ the hyperelliptic mapping class group $\cM^h(N)$ coincide with the full mapping class group $\cM(N)$ (see Corollary \ref{Presen:g3:Bir}). If $g\geq 4$, then Paris and Szepietowski \cite{SzepParis} obtained a simple presentation of $\cM(N)$, which can be rewritten (Proposition 3.3 and Theorem 3.5 of \cite{StukowSimpSzepPar}) so that it has the hyperelliptic involution $\ro$ as one of the generators, and the hyperelliptic relations (Theorem~\ref{th:004}) appear among defining relations. 

As an application of obtained presentations we compute the first homology groups of $\cM^h(N)$ and $\cM^{h+}(N)$ with 
coefficients in $H_1(N;\zz)$ 
(Theorems \ref{tw:hom:positive} and \ref{tw:hom:full}).
\section{Definitions of $\cM^h(N_g)$ and $\cM^{h+}(N_g)$} \label{sec:def:hiper} %
Let $S_{g-1}$  be a closed oriented surface of genus $g-1\geqslant 2$ embedded in $\rr^3$ as shown in
Figure \ref{r01}, in such a way that it is invariant under reflections across $xy,yz,xz$ planes, and
let $\map{j}{S_{g-1}}{S_{g-1}}$ be the symmetry defined by $j(x,y,z)=(-x,-y,-z)$. Denote by $C_{\cM^{\pm}(S_{g-1})}(j)$
the centraliser of $j$ in $\cM^{\pm}(S_{g-1})$. The orbit space $S_{g-1}/\gen{j}$ is a nonorientable
surface $N_{g}$ of genus $g$ and it is known (Theorem 1 of \cite{BirChil1}) that there is an epimorphism 
\[\map{\pi_j}{C_{\cM^{\pm}(S_{g-1})}(j)}{\cM(N_{g})}\]
with kernel $\ker \pi_j=\gen{j}$. In particular
\[\cM(N_{g})\cong C_{\cM^{\pm}(S_{g-1})}(j)/\gen{j}.\]
Observe that the hyperelliptic involution $\ro$ is an element of $C_{\cM^{\pm}(S_{g-1})}(j)$. Hence the
following definition makes sense.
\begin{df}
 Define the \emph{hyperelliptic mapping class group} $\cM^h(N)$ of a closed nonorientable surface $N$ to be
the centraliser of $\pi_j(\ro)$ in the mapping class group $\cM(N)$. We say that $\pi_j(\ro)$ is the \emph{hyperelliptic
involution} of $N$ and by abuse of notation we write $\ro$ for $\pi_j(\ro)$.
\end{df}
In order to have a little more straightforward description of $\ro$ observe, that the orbit space $S_{g-1}/\gen{j}$ gives the
model of $N_g$, where $N_g$ is a connected sum of an orientable surface $S_r$ and a projective plane (for $g$ odd) or a Klein
bottle (for $g$ even) -- see Figure \ref{r02}. 
\begin{figure}[h]
\begin{center}
\includegraphics[width=0.96\textwidth]{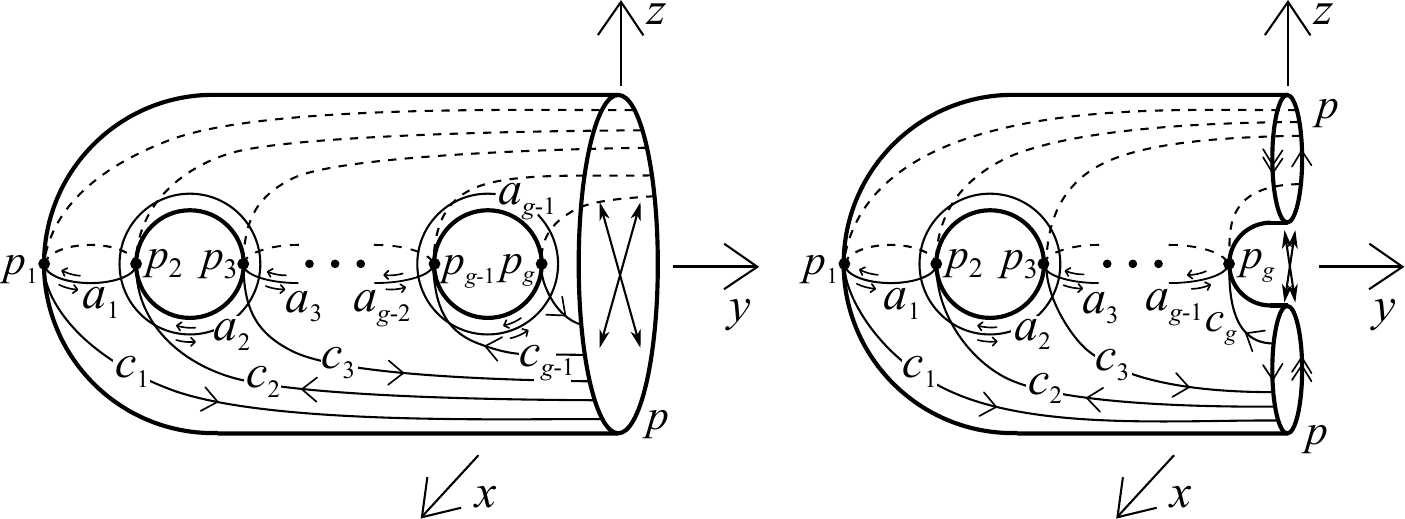}
\caption{Nonorientable surface $N_g$.}\label{r02} %
\end{center}
\end{figure}
To be more precise, $N_g$ is the left half of $S_{g-1}$ embedded in $\rr^3$ as
in Figure \ref{r01} with boundary points identified by the map $(x,y,z)\mapsto (-x,-y,-z)$. Note that $g=2r+1$ for $g$ odd and
$g=2r+2$ for $g$ even. In such a model, $\map{\ro}{N_g}{N_g}$ is the map induced by the half turn about the $y$-axis. 

Observe that the set of fixed points of $\map{\ro}{N_g}{N_g}$ consists of $g$ points $\{p_1,p_2,\ldots,p_g\}$ and the circle
$p$. Therefore $\cM^h(N)$ consists of isotopy classes of maps which must fix the set $\{p_1,p_2,\ldots,p_g\}$ and map the 
circle $p$ to itself.
Moreover, the orbit space $N_g/\gen{\ro}$ is the sphere $S_{0,1}^g$ with one boundary component corresponding to $p$ and
$g$ distinguished points corresponding to $\{p_1,p_2,\ldots,p_g\}$. Since elements of $\cM^h(N_g)$ may not fix
$p$ point--wise, it is more convenient to treat $p$ as the distinguished puncture $p_{g+1}$, hence we will
identify $N_g/\gen{\ro}$ with the sphere $S_0^{g,1}$ with $g+1$ punctures. The notation $S_0^{g,1}$ is meant to indicate
that maps of $S_0^{g,1}$ (and their isotopies) could permute the punctures $p_1,\ldots,p_g$, but must fix $p_{g+1}$. 

 The main goal of this section is to prove the following theorem.
\begin{tw}\label{tw:002}
 If $g\geqslant 3$ then the projection $N_g\to N_g/\gen{\ro}$ induces an epimorphism
\[\map{\pi_\ro}{\cM^h(N_g)}{\cM^{\pm}(S_0^{g,1})}\]
with $\ker \pi_\ro=\gen{\ro}$.
\end{tw}
\begin{proof}
 Consider the following diagram
\[\xymatrix@C=3pc@R=3pc{
 C_{\cM^{\pm}(S_{g-1})}(\gen{j,\ro})\ar^{\pi_\ro}[r]\ar^{\pi_j}[d]&C_{\cM^{\pm}(S_0^{2g})}(j)\ar^{\pi_j}[d]\\
 \cM^h(N_g)\ar@{-->}^{\pi_\ro}[r] \ar@/^1pc/@{->}[u]^{i_j} &\cM^{\pm}(S_0^{g,1})
}\]
 The left vertical map is the restriction of the projection 
 \[\map{\pi_j}{C_{\cM^{\pm}(S_{g-1})}(j)}{\cM(N_g)}\] to the subgroup consisting of elements which
 centralise $\ro$. The nice thing about $\pi_j$ is that it has a section
 \[\map{i_j}{\cM(N_g)}{C_{\cM^{\pm}(S_{g-1})}(j)}.\]
 In fact, for any $h\in\cM(N_g)$ we can define $i_j(h)$ to be an orientation preserving lift of $h$. 
 
 The upper horizontal map is the restriction of the homomorphism
 \[\map{\pi_\ro}{\cM^{h\pm}(S_{g-1})}{\cM^{\pm}(S_0^{2g})}\]
induced by the orbit projection $S_{g-1}\to S_{g-1}/\gen{\ro}$. The fact that this map is a
homomorphism was
first observed by Birman and Hilden \cite{Bir-Hil}. The kernel of this map is
equal to $\gen{\ro}$.

The right vertical map is again the homomorphism induced by the orbit projection $S^{2g}_0\to
S^{2g}_0/\gen{j}$. However now $\map{j}{S^{2g}_0}{S^{2g}_0}$ is a reflection with a circle of fixed points. The existence of
$\pi_j$ in such a case follows from the work of Zieschang (Proposition 10.3 of \cite{Zies1}).

Hence there is the homomorphism 
\[\map{\pi_\ro}{\cM^h(N_g)}{\cM^{\pm}(S_0^{g,1})}\]
defined as the composition
\[\pi_\ro=\pi_j\circ \pi_\ro\circ i_j.\]
Moreover,
\[\begin{aligned}
\ker \pi_\ro&=\ker (\pi_j\circ \pi_\ro\circ i_j)=(\pi_j\circ \pi_\ro\circ i_j)^{-1}(id)\\
&=i_j^{-1}(\pi_\ro^{-1}(\pi_j^{-1}(id)))=i_j^{-1}(\pi_\ro^{-1}(\gen{j}))=i_j^{-1}(\gen{j,\ro})=\gen{\ro}.
\end{aligned}\]
\end{proof}
\begin{uw}
Theorem \ref{tw:002} is not true if $N=N_2$. This corresponds to the fact that the Birman-Hilden theorem does not hold for the
closed torus $S=S_1$.
\end{uw}
\begin{uw}\label{rem:positiv:subg}
Theorem \ref{tw:002} shows that the group $\cM^h(N_g)$ contains a very natural subgroup of index 2, namely 
\[\cM^{h+}(N_g)=\pi_\ro^{-1}(\cM(S_0^{g,1})).\]
Geometrically, the subgroup $\cM^{h+}(N_g)$ consists of these elements, which preserve the orientation of the circle $p$
(the circle fixed by $\ro$). As we will see later (see Remark \ref{uw:perspective}), it seems that
the group $\cM^{h+}(N)$ corresponds to $\cM^h(S)$, whereas $\cM^{h}(N)$ corresponds to $\cM^{h\pm}(S)$.
\end{uw}
\section{Presentations for groups $\cM(S_0^{g,1})$ and $\cM^{\pm}(S_0^{g,1})$}%
Let $w_1,w_2,\ldots,w_g$ be simple arcs connecting punctures $p_1,\ldots,p_{g+1}$ on a sphere $S_0^{g+1}$ as shown in Figure
\ref{r03}. 
\begin{figure}[h]
\begin{center}
\includegraphics[width=0.8\textwidth]{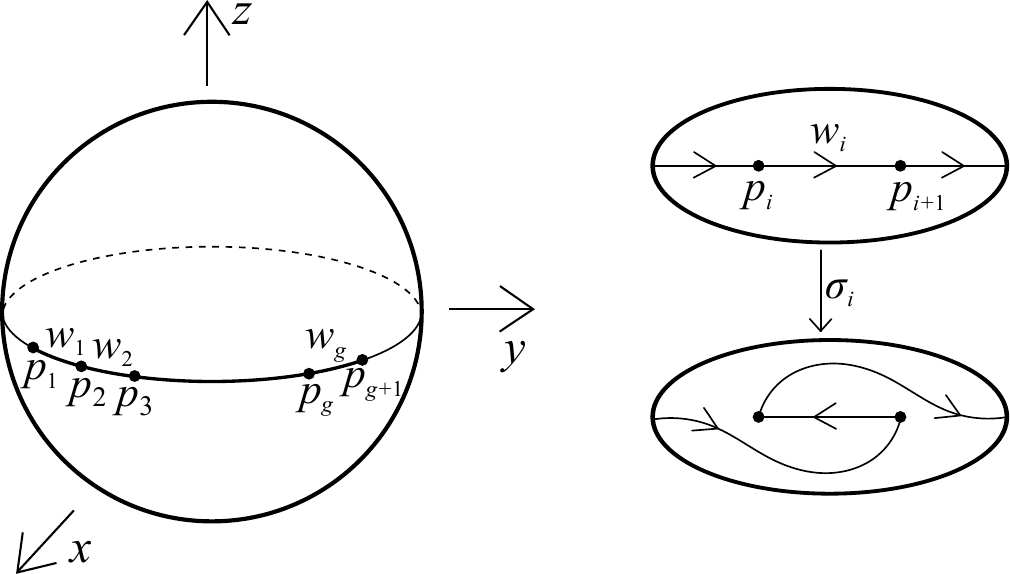}
\caption{Sphere $S_0^{p+1}$ and elementary braid $\sig_i$.}\label{r03} %
\end{center}
\end{figure}
Recall that to each such arc $w_i$ we can associate the elementary braid $\sig_i$ which interchanges punctures $p_i$ and
$p_{i+1}$ -- see Figure \ref{r03}. The following theorem is due to Magnus \cite{Mag_sfer}. It is also proved in Chapter 4 of
\cite{Bir-BLMCG}.
\begin{tw}\label{th:00}
 If $g\geqslant 1$, then $\cM(S_0^{g+1})$ has the presentation with generators $\sig_1,\ldots,\sig_g$ and defining relations:
\[\begin{aligned}
   &\sig_k\sig_j=\sig_j\sig_k\quad\text{for $|k-j|>1$},\\
&\sig_j\sig_{j+1}\sig_j=\sig_{j+1}\sig_j\sig_{j+1}\quad\text{for $j=1,\ldots,g-1$},\\
&\sig_1\cdots \sig_{g-1}\sig_g^2\sig_{g-1}\cdots \sig_1=1,\\
&(\sig_1\sig_2\cdots \sig_{g})^{g+1}=1.
  \end{aligned} 
\]\end{tw}
In order to avoid unnecessary complications, from now on assume that $g\geqslant 3$.
Recall that we denote by $\cM(S_0^{g,1})$ the subgroup of $\cM(S_0^{g+1})$ consisting of maps which fix $p_{g+1}$. 
\begin{tw}\label{th:001}
 If $g\geqslant 3$, then $\cM(S_0^{g,1})$ has the presentation with generators $\sig_1,\ldots,\sig_{g-1}$ and defining
relations:
\begin{enumerate}
 \item[(A1)] $\sig_k\sig_j=\sig_j\sig_k$ for $|k-j|>1$ and $k,j<g$,
\item[(A2)] $\sig_j\sig_{j+1}\sig_j=\sig_{j+1}\sig_j\sig_{j+1}$ for $j=1,2,\ldots,g-2$,
\item[(A3)] $(\sig_1\cdots \sig_{g-1})^g=1$.
\end{enumerate}
\end{tw}
\begin{proof}
 By Lemma 2.2 of \cite{Big-lin}, 
\[\cM(S_0^{g,1})\cong B_{g}/\gen{\Delta^2},\]
where $B_{g}=\cM(S_{0,1}^{g})$ is the braid group on $g$ strands, and
\[\Delta^2=(\sig_1\cdots \sig_{g-1})^g\]
is the generator of the center of $B_g$. Since $B_g$ has the presentation with generators $\sig_1,\ldots,\sig_{g-1}$ and
defining relations (A1), (A2), this completes the proof.
\end{proof}
\begin{uw}
 Theorem \ref{th:001} can be also algebraically deduced from Theorem \ref{th:00}. Since $\cM(S_0^{g,1})$
is a subgroup of index $g+1$ in $\cM(S_0^{g+1})$, for the Schreier transversal we can take 
\[(1,\sig_{g},\sig_{g}\sig_{g-1},\ldots,\sig_{g}\sig_{g-1}\cdots \sig_1).\]
If we now apply Reidemeister-Schreier process, as generators for $\cM(S_0^{g,1})$ we get 
$\sig_1,\ldots,\sig_{g-1}$ and
additionally $\tau_1,\ldots,\tau_g$ where
\[\tau_k=\begin{cases}
          \sig_g\cdots\sig_{k+1}\sig_k^2\sig_{k+1}^{-1}\cdots\sig_g^{-1}&\text{for $k=1,\ldots,g-1$}\\
\sig_g^2&\text{for $k=g$}.
         \end{cases}\]
As defining relations we get
\begin{enumerate}
 \item[] $\sig_k\sig_j=\sig_j\sig_k$ for $|k-j|>1$ and $k,j<g$,
 \item[] $\sig_k\tau_j=\tau_j\sig_k$ for $j\neq k,k+1$,
\item[] $\sig_j\sig_{j+1}\sig_j=\sig_{j+1}\sig_j\sig_{j+1}$ for $j=1,2,\ldots,g-2$,
\item[] $\sig_k\tau_{k+1}\sig_{k}^{-1}=\tau_{k+1}^{-1}\tau_k\tau_{k+1}$ for $k=1,2,\ldots,g-1$,
\item[] $\sig_k\tau_k\sig_{k}^{-1}=\tau_{k+1}$ for $k=1,2,\ldots,g-1$,
\item[] $\tau_1\tau_2\cdots \tau_g=1$,
\item[] $\sig_{g-1}\cdots\sig_2\sig_1\tau_1\sig_1\sig_2\cdots \sig_{g-1}=1$,
\item[] $\left(\sig_{g-1}\sig_{g-2}\cdots\sig_{1}\tau_1\right)^g=1$.
\end{enumerate}
If we now remove generators $\tau_1,\ldots,\tau_g$ from the above presentation, we obtain the presentation given
by Theorem \ref{th:001}. The computations are lengthy, but completely straightforward. 
\end{uw}
Recall that by $\cM^{\pm}(S_0^{g,1})$ we denote the extended mapping class group of the sphere $S_0^{g,1}$, that is the
extension of degree 2 of $\cM(S_0^{g,1})$. Suppose that the sphere $S_0^{g,1}$ is the metric sphere in $\rr^3$ with origin
$(0,0,0)$ and that punctures $p_1,\ldots,p_g$ are contained in the $xy$-plane. Let $\map{\sig}{{S_0^{g,1}}}{{S_0^{g,1}}}$ be
the map induced by the reflection across the $xy$-plane. We have the short exact sequence.
\[\xymatrix{
1\ar[r]&\cM(S_0^{g,1})\ar[r]&\cM^{\pm}(S_0^{g,1})\ar[r]&\gen{\sig}\ar[r]&1.
}\]
 Moreover, $\sig\sig_i\sig^{-1}=\sig_i^{-1}$ for $i=1,\ldots,g-1$. Therefore Theorem \ref{th:001} implies the following.
\begin{tw} \label{th:003}
 If $g\geqslant 3$, then $\cM^{\pm}(S_0^{g,1})$ has the presentation with generators $\sig_1,\ldots,\sig_{g-1},\sig$ and
defining relations:
\begin{enumerate}
 \item[(B1)] $\sig_k\sig_j=\sig_j\sig_k$ for $|k-j|>1$ and $k,j<g$,
\item[(B2)] $\sig_j\sig_{j+1}\sig_j=\sig_{j+1}\sig_j\sig_{j+1}$ for $j=1,2,\ldots,g-2$,
\item[(B3)] $(\sig_1\cdots \sig_{g-1})^g=1$,
\item[(B4)] $\sig^2=1$,
\item[(B5)] $\sig\sig_i\sig=\sig_i^{-1}$ for $i=1,2,\ldots,g-1$.
\end{enumerate}
\end{tw}
\section{Presentations for groups $\cM^h(N_g)$ and $\cM^{h+}(N_g)$}
By Theorem \ref{tw:002} there is a short exact sequence.
\[\xymatrix{
1\ar[r]&\gen{\ro}\ar[r]&\cM^h(N_g)\ar[r]^{\pi_\ro}&\cM^{\pm}(S_0^{g,1})\ar[r]&1.
}\]
Moreover, it is known that as lifts of braids $\sig_1,\ldots,\sig_{g-1}\in\cM^{\pm}(S_0^{g,1})$ we can take Dehn twists
$t_{a_1},\ldots,t_{a_{g-1}}\in \cM^h(N_g)$ about circles $a_1,\ldots,a_{g-1}$ -- cf Figure \ref{r02} (small arrows in this
picture indicate directions of twists). As a
lift of $\sig$ we take the symmetry $s$ across the $xy$ plane (the second lift of $\sig$ is the symmetry $\ro s$, that is
the symmetry across the $yz$ plane). 

To obtain a presentation for the group $\cM^h(N_g)$ we need to
lift relations (B1)--(B5) of Theorem \ref{th:003}. Each relation of the form \[w(\sig_1,\ldots,\sig_{g-1},\sig)=1\] lifts
either to $w(t_{a_1},\ldots,t_{a_{g-1}},s)=1$ or to $w(t_{a_1},\ldots,t_{a_{g-1}},s)=\ro$. In order to determine which of
these two cases does occur it is enough to check whether the homeomorphism $w(t_{a_1},\ldots,t_{a_{g-1}},s)$ changes the orientation of
the circle $a_1$ or not. This can be easily done and as a result we obtain the following theorem.
\begin{tw} \label{th:004}
 If $g\geqslant 3$, then $\cM^h(N_g)$ has the presentation with generators $t_{a_1},\ldots,t_{a_{g-1}},s,\ro$ and defining
relations:
\begin{enumerate}
 \item[(C1)] $t_{a_k}t_{a_j}=t_{a_j}t_{a_k}$ for $|k-j|>1$ and $k,j<g$,
\item[(C2)] $t_{a_j}t_{a_{j+1}}t_{a_j}=t_{a_{j+1}}t_{a_j}t_{a_{j+1}}$ for $j=1,2,\ldots,g-2$,
\item[(C3)] $(t_{a_{1}}\cdots t_{a_{g-1}})^g=\begin{cases}
                                              1&\text{for $g$ even}\\
\ro&\text{for $g$ odd,}\\\end{cases}$ 
\item[(C4)] $s^2=1$,
\item[(C5)] $st_{a_j}s=t_{a_j}^{-1}$ for $j=1,2,\ldots,g-1$,
\item[(C6)] $\ro^2=1$,
\item[(C7)] $\ro t_{a_j}\ro=t_{a_j}$ for $j=1,2,\ldots,g-1$,
\item[(C8)] $\ro s\ro=s$. 
\end{enumerate}
\end{tw}
\begin{wn}\label{cr:001}
 If $g\geqslant 3$, then 
\[H_1(\cM^h(N_g))=\begin{cases}
                   \zz_2\oplus\zz_2&\text{for $g$ odd}\\
\zz_2\oplus\zz_2\oplus\zz_2&\text{for $g$ even.}
                  \end{cases}
\]
\end{wn} 
\begin{proof}
Relation (C2) implies that the abelianization of the group $\cM^h(N_g)$ is an abelian group generated by $t_{a_1},s,\ro$.
Defining relations take form
\[\begin{aligned}
   &t_{a_1}^{(g-1)g}=\begin{cases}
            1&\text{for $g$ even}\\
\ro&\text{for $g$ odd}\\\end{cases}\\
&s^2=1,\quad t_{a_1}^2=1,\quad \ro^2=1.
  \end{aligned}
\]
Hence $H_1(\cM^h(N_g))=\gen{t_{a_1},s}\simeq \zz_2\oplus\zz_2$ for $g$ odd and $H_1(\cM^h(N_g))=\gen{t_{a_1},s,\ro}\simeq
\zz_2\oplus\zz_2\oplus\zz_2$ for $g$ even.
\end{proof}
The main theorem of \cite{Chil} implies that the group $\cM(N_3)$ is generated by $a_1,a_2$ and a crosscap slide which commutes with $\ro$. Hence $\cM^h(N_3)=\cM(N_3)$ and Theorem \ref{th:004} implies the following.
\begin{wn}[Birman-Chillingworth \cite{BirChil1}] \label{Presen:g3:Bir}
 The group $\cM(N_3)$ has the presentation with generators $t_{a_1},t_{a_2},s$ and defining relations:
\begin{enumerate}
 \item[(D1)] $t_{a_1}t_{a_2}t_{a_1}=t_{a_2}t_{a_1}t_{a_2}$,
\item[(D2)] $(t_{a_1}t_{a_2}t_{a_1})^4=1$,
\item[(D3)] $s^2=1$,
\item[(D4)] $st_{a_j}s=t_{a_j}^{-1}$ for $j=1,2$.
\end{enumerate}
\end{wn}
\begin{proof}
 By Theorem \ref{th:004}, the group $\cM(N_3)$ is generated by $t_{a_1},t_{a_2},\ro,s$ with defining relations:
\begin{enumerate}
\item[(C2)] $t_{a_1}t_{a_{2}}t_{a_1}=t_{a_{2}}t_{a_1}t_{a_{2}}$,
\item[(C3)] $(t_{a_{1}}t_{a_2})^3=\ro$, 
\item[(C4)] $s^2=1$,
\item[(C5)] $st_{a_j}s=t_{a_j}^{-1}$ for $j=1,2$,
\item[(C6)] $\ro^2=1$,
\item[(C7)] $\ro t_{a_j}\ro=t_{a_j}$ for $j=1,2$,
\item[(C8)] $\ro s\ro=s$. 
\end{enumerate}
Using (C2), we can rewrite (C3) in the form
\[\begin{aligned}
 \ro=t_{a_{1}}t_{a_2}t_{a_{1}}(t_{a_2}t_{a_{1}}t_{a_2})=t_{a_{1}}t_{a_2}t_{a_{1}}(t_{a_1}t_{a_{2}}t_{a_1
})=(t_{a_1}t_{a_{2}}t_{a_1})^2.
  \end{aligned}
\]
Hence we can remove $\ro$ from the generating set and then (C6) will transform into (D2). It remains to check that relations
(C7) and (C8) are superfluous. Let start with (C7).
\[\begin{aligned}
t_{a_1}\ro t_{a_1}^{-1}&=t_{a_1}(t_{a_1}t_{a_{2}}t_{a_1})(t_{a_1}t_{a_{2}}t_{a_1})t_{a_1}^{-1}\\
&=t_{a_1}(t_{a_2}t_{a_{1}}t_{a_2})(t_{a_1}t_{a_{2}}t_{a_1})t_{a_1}^{-1}=
(t_{a_1}t_{a_2}t_{a_{1}})(t_{a_1}t_{a_2}t_{a_{1}})=\ro\\
t_{a_2}\ro t_{a_2}^{-1}&=t_{a_2}(t_{a_1}t_{a_{2}}t_{a_1})(t_{a_1}t_{a_{2}}t_{a_1})t_{a_2}^{-1}\\
&=t_{a_2}(t_{a_1}t_{a_{2}}t_{a_1})(t_{a_2}t_{a_{1}}t_{a_2})t_{a_2}^{-1}=
(t_{a_1}t_{a_2}t_{a_{1}})(t_{a_1}t_{a_2}t_{a_{1}})=\ro
  \end{aligned}\]
Now we check (C8).
\[\begin{aligned}  
s\ro s&=
s(t_{a_1}t_{a_{2}}t_{a_1})^2s=(t_{a_1}^{-1}t_{a_{2}}^{-1}t_{a_1}^{-1})^2=(t_{a_1}t_{a_2}t_{a_1})^{-2
} = (t_{a_1}t_{a_2}t_{a_1})^{2}=\ro.
  \end{aligned}\]
\end{proof}
By restricting homomorphism $\map{\pi_\ro}{\cM^{h}(N_g)}{\cM^{\pm}(S_0^{g,1})}$ to the subgroup $\cM^{h+}(N_g)$ we obtain the
exact sequence
\[\xymatrix{
1\ar[r]&\gen{\ro}\ar[r]&\cM^{h+}(N_g)\ar[r]^{\pi_\ro}&\cM(S_0^{g,1})\ar[r]&1.
}\]
Now if we lift the presentation from Theorem \ref{th:001}, we get the following.
\begin{tw}\label{tw:pres:pos}
 If $g\geqslant 3$, then $\cM^{h+}(N_g)$ has the presentation with generators $t_{a_1},\ldots,t_{a_{g-1}},\ro$ and defining
relations:
\begin{enumerate}
 \item[(E1)] $t_{a_k}t_{a_j}=t_{a_j}t_{a_k}$ for $|k-j|>1$ and $k,j<g$,
\item[(E2)] $t_{a_j}t_{a_{j+1}}t_{a_j}=t_{a_{j+1}}t_{a_j}t_{a_{j+1}}$ for $j=1,2,\ldots,g-2$,
\item[(E3)] $(t_{a_{1}}\cdots t_{a_{g-1}})^g=\begin{cases}
                                              1&\text{for $g$ even}\\
\ro&\text{for $g$ odd,}\\\end{cases}$
\item[(E4)] $\ro^2=1$,
\item[(E5)] $\ro t_{a_j}\ro=t_{a_j}$ for $j=1,2,\ldots,g-1$. 
\end{enumerate}
\end{tw}
\begin{wn}\label{cr:002}
 If $g\geqslant 3$, then 
\[H_1(\cM^{h+}(N_g))=\begin{cases}
                   \zz_{2(g-1)g}&\text{for $g$ odd}\\
\zz_{(g-1)g}\oplus\zz_2&\text{for $g$ even.}
                  \end{cases}
\]
\end{wn} 
\begin{proof}
Relation (E2) implies that the abelianization of the group $\cM^{h+}(N_g)$ is an abelian group generated by $t_{a_1},\ro$.
Defining relations take form:
\[\begin{aligned}
   &t_{a_1}^{(g-1)g}=\begin{cases}
            1&\text{for $g$ even}\\
\ro&\text{for $g$ odd,}\\\end{cases}\\
&\ro^2=1.
  \end{aligned}
\]
Hence $H_1(\cM^{h+}(N_g))=\gen{t_{a_1}}\simeq \zz_{2(g-1)g}$ for $g$ odd and
$H_1(\cM^{h+}(N_g))=\gen{t_{a_1},\ro}\simeq
\zz_{(g-1)g}\oplus\zz_2$ for $g$ even.
\end{proof}
\begin{uw}\label{uw:perspective}
 To put Corollaries \ref{cr:001} and \ref{cr:002} into perspective, recall that in the oriented case (Theorem 8 of \cite{Bir-Hil}),
\[\begin{aligned}
   \cM^{h}(S_g)=\langle& t_{a_1},\ldots,t_{a_{2g+1}},\ro\st 
t_{a_k}t_{a_j}=t_{a_j}t_{a_k},\ t_{a_j}t_{a_{j+1}}t_{a_j}=t_{a_{j+1}}t_{a_j}t_{a_{j+1}},\\
 &(t_{a_1}t_{a_2}\cdots t_{a_{2g+1}})^{2g+2}=1,\ \ro=t_{a_1}t_{a_2}\cdots t_{a_{2g+1}}t_{a_{2g+1}}\cdots t_{a_2}t_{a_1},\\ &\ro^2=1, \ro t_{a_1}\ro=t_{a_1}\rangle,\quad
\text{where $j=1,2,\ldots,2g$, $|k-j|>1$}.
  \end{aligned}
\]
The presentation for the group $\cM^{h\pm}(S_g)$ is obtained from the above presentation by adding one generator $s$ and three relations: 
\[s^2=1,\ st_{a_1}s=t_{a_1}^{-1},\ \ro s\ro=s.\]
Consequently,
$H_1(\cM^{h\pm}(S_g))=\zz_2\oplus\zz_2$ and \[H_1(\cM^{h}(S_g))=\begin{cases}
            \zz_{4g+2}&\text{for $g$ even}\\
\zz_{8g+4}&\text{for $g$ odd.}\\\end{cases}\]
This suggests that algebraically the group $\cM^{h+}(N)$ corresponds to $\cM^h(S)$, whereas $\cM^{h}(N)$ corresponds to
$\cM^{h\pm}(S)$.
\end{uw}
\section{Computing $H_1(\cM^{h+}(N_g); H_1(N_g;\zz))$ and $H_1(\cM^{h}(N_g); H_1(N_g;\zz))$}
\subsection{Homology of groups}
Let us briefly review how to compute the first homology of a group with twisted coefficients. Our exposition follows
\cite{Brown,Tanaka}.

For a given group $G$ and $G$-module $M$ (that is $\zz G$-module) we define the \emph{bar resolution} which
is a chain complex $(C_n(G))$ of $G$-modules, where $C_n(G)$ is the free $G$-module generated by symbols $[h_1|\cdots|h_n]$,
$h_i\in G$. For $n=0$, $C_0(G)$ is the free module generated by the empty bracket $[\cdot]$. Our interest will restrict to
groups $C_2(G),C_1(G),C_0(G)$ for which the boundary operator $\map{\partial_n}{C_n(G)}{C_{n-1}(G)}$ is defined by formulas:
\[\begin{aligned}
\partial_2([h_1|h_2])&=h_1[h_2]-[h_1h_2]+[h_1],\\
   \partial_1([h])&=h[\cdot]-[\cdot].
  \end{aligned}
\]

The homology of $G$ with coefficients in $M$ is defined as the homology groups of the chain complex $(C_n(G)\otimes M)$, where
the chain complexes are tensored over $\zz G$.  In particular, $H_1(G;M)$ is the first homology group of the
complex
\[\xymatrix@C=3pc@R=3pc{C_2(G)\otimes M\ar[r]^{\ \partial_2\otimes {\rm id}}&C_1(G)\otimes M
\ar[r]^{\ \partial_1\otimes {\rm id}}&C_0(G)\otimes M.}\]
For simplicity, we denote $\partial\otimes {\rm id}=\kre{\partial}$ henceforth.

If the group $G$ has a presentation $G=\langle X\,|\,R\rangle$, denote by
\[\gen{\kre{X}}=\gen{[x]\otimes m\st x\in X, m\in M}\podz C_1(G)\otimes M.\]
Then, using the formula for $\partial_2$, one can show that $H_1(G;M)$ is a quotient of  $\gen{\kre{X}}\cap \ker\kre{\partial}_1$.

The kernel of this quotient corresponds to relations in $G$
(that is elements of $R$). To be more precise, if
$r\in R$ has the form $x_1\cdots x_k=y_1\cdots y_n$ and $m\in M$, then 
$r$ gives the relation (in $H_1(G;M)$)
\begin{equation}
 \kre{r}\otimes m\!:\ \sum_{i=1}^{k}x_1\cdots x_{i-1}[x_i]\otimes m=\sum_{i=1}^{n}y_1\cdots y_{i-1}[y_i]\otimes m.\label{eq_rew_rel}
\end{equation}
Then 
\[H_1(G;M)=\gen{\kre{X}}\cap \ker\kre{\partial}_1/\gen{\kre{R}},\]
where 
\[\kre{R}=\{\kre{r}\otimes m\st r\in R,m\in M\}.\]
\subsection{Action of $\cM^{h}(N_g)$ on $H_1(N_g;\zz)$}
Let $c_1,\ldots,c_g$ be one-sided circles indicated in Figure \ref{r04}. In this figure surface $N_g$ is represented as
the sphere with $g$ crosscaps (the shaded disks represent crosscaps, hence their interiors are to be
removed and then the antipodal points on each boundary component are to be identified). 
\begin{figure}[h]
\begin{center}
\includegraphics[width=0.8\textwidth]{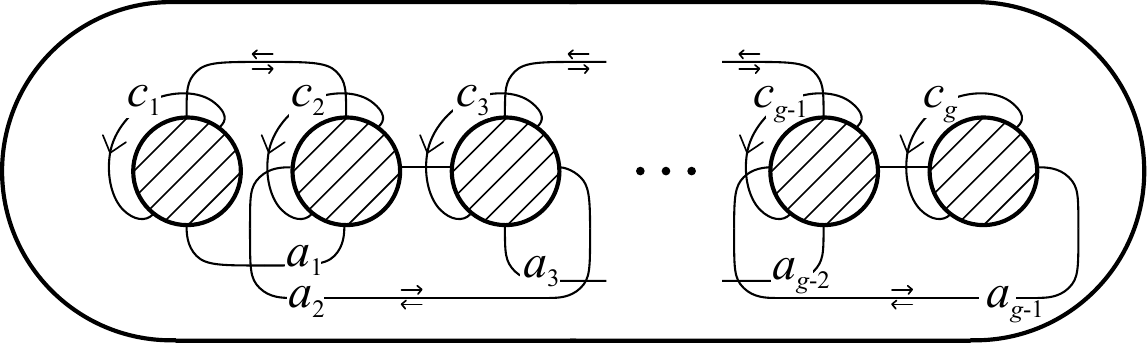}
\caption{Surface $N_g$ as a sphere with crosscaps.}\label{r04} %
\end{center}
\end{figure}
The same set of circles is also indicated in Figure \ref{r02} -- for a method of transferring circles between two
models of
$N_g$ see Section 3 of \cite{Stukow_HomTw}.

Recall that $H_1(N_g;\zz)$ as a $\zz$-module is
generated by $\gamma_1=[c_1],\ldots,\gamma_g=[c_g]$ with respect to the single relation
\[2(\gamma_1+\gamma_2+\cdots+\gamma_g)=0.\]
There is a $\zz_2$-valued intersection paring $\langle\,,\rangle$ on $H_1(N_g;\zz)$ defined as the symmetric bilinear form
(with values in $\zz_2$) satisfying $\langle\gamma_i,\gamma_j\rangle=\delta_{ij}$ for $1\leqslant i,j\leqslant g$. The mapping class group
$\cM(N_g)$ acts on
$H_1(N_g;\zz)$ via automorphisms which preserve $\langle\,,\rangle$, hence there is a representation 
\[\map{\psi}{\cM(N_g)}{{\rm Iso}(H_1(N_g;\zz))}.\]
In fact it is known that this representation is surjective -- see \cite{Pinkal,Gadgil}.

Since we have very simple geometric definitions of $t_{a_i},s,\ro\in \cM^h(N_g)$ it is straightforward to check that
\[\begin{aligned}
   \psi(t_{a_i})&=I_{i-1}\oplus \begin{bmatrix}
0&1\\
-1&2\end{bmatrix}\oplus I_{g-i-1}\\
\psi(t_{a_i}^{-1})&=I_{i-1}\oplus \begin{bmatrix}
2&-1\\
1&0
    \end{bmatrix}\oplus I_{g-i-1}\\
\psi(s)&=\begin{bmatrix}
          -1&2&-2&2&\ldots&(-1)^{g}\cdot 2\\
0&1&-2&2&\ldots&(-1)^{g}\cdot 2\\
0&0&-1&2&\ldots&(-1)^{g}\cdot 2\\
0&0&0&1&\ldots&(-1)^{g}\cdot 2\\
\vdots&\vdots&\vdots&\vdots&\ddots&\vdots\\
0&0&0&0&\ldots& (-1)^{g}\cdot 1
         \end{bmatrix}
\\
\psi(\ro)&=-I_g
  \end{aligned}
\] 
where $I_k$ is the identity matrix of rank $k$.

The above matrices are written with respect to the generating set $(\gamma_1,\gamma_2,\ldots,\gamma_g)$. Note that
$H_1(N_g;\zz)$ is
not free, hence one has to be careful with matrices -- two different matrices may represent the same element.
\subsection{Computing $\gen{\kre{X}}\cap \ker \kre{\partial}_1$} 
Observe that if $G=\cM^h(N_g)$, $M=H_1(N_g;\zz)$ and $h\in G$ then
\[\kre{\partial}_1([h]\otimes \gamma_j)=(h-1)[\cdot]\otimes \gamma_j=[\cdot]\otimes(\psi(h)^{-1}-I_g)\gamma_j.\]
If we identify $C_0(G)\otimes M$ with $M$ by the map $[\cdot]\otimes m\mapsto m$, this formula takes form
\[\kre{\partial}_1([h]\otimes \gamma_j)=(\psi(h)^{-1}-I_g)\gamma_j.\]
Let us denote $[\ro]\otimes \gamma_j, [s]\otimes \gamma_j,
[t_{a_i}]\otimes \gamma_j$ respectively by $\ro_j,s_j$ and $t_{i,j}$.
 Using the above formula, we obtain
\[\begin{aligned}
\kre{\partial}_1(\ro_j)&=-2\gamma_j\\
\kre{\partial}_1(s_j)&=\begin{cases}
-2\sum_{k=1}^{j}\gamma_k &\text{for $j$ odd}\\
-\kre{\partial}_1(s_{j-1}) &\text{for $j$ even}
                       \end{cases}
\\
   \kre{\partial}_1(t_{i,j})&=\begin{cases}
                              \gamma_i+\gamma_{i+1}&\text{for $j=i$}\\
-\gamma_i-\gamma_{i+1}&\text{for $j=i+1$}\\
0&\text{otherwise.}
                             \end{cases}
  \end{aligned}
\] 
\begin{prop}
Let $g\geqslant 3$ and $G=\cM^{h+}(N_g)$ then $\gen{\kre{X}}\cap\ker\kre{\partial}_1$ is the abelian group which admits 
the presentation with generators:
\begin{enumerate}
 \item[(F1)] $t_{i,j}$, where $i=1,\ldots,g-1$ and $j=1,\ldots,i-1,i+2,\ldots,g$
\item[(F2)] $t_{j,j}+t_{j,j+1}$, where $j=1,\ldots,g-1$
\item[(F3)] $2t_{j,j}+\ro_j+\ro_{j+1}$ , where $j=1,\ldots,g-1$
\item[(F4)] $\begin{cases}
 2t_{1,1}+2t_{3,3}+\cdots+2t_{g-2,g-2}-\ro_g&\text{for $g$ odd}\\
2t_{1,1}+2t_{3,3}+\cdots+2t_{g-1,g-1}&\text{for $g$ even}
      \end{cases}$
\end{enumerate}
and relations
\[\begin{aligned}
r_{t_j}\!&:0=2t_{j,1}+\cdots +2(t_{j,j}+t_{j,j+1})+\cdots+2t_{j,g}\quad\text{for $j=1,\ldots,g-1$}\\
r_{\ro}\!&:\begin{cases}
              2(2t_{1,1}+\ro_1+\ro_2)+\cdots+2(2t_{g-2,g-2}+\ro_{g-2}+\ro_{g-1})\\
\qquad\qquad =2(2t_{1,1}+2t_{3,3}+\cdots+2t_{g-2,g-2}-\ro_g)\quad \text{for $g$ odd}\\
              2(2t_{1,1}+\ro_1+\ro_2)+\cdots+2(2t_{g-1,g-1}+\ro_{g-1}+\ro_{g})\\
\qquad\qquad\qquad =2(2t_{1,1}+2t_{3,3}+\cdots+2t_{g-1,g-1})\quad \text{for $g$ even.}
             \end{cases}
\end{aligned}
\]
\end{prop}
\begin{proof}
By Theorem \ref{tw:pres:pos}, $\gen{\kre{X}}$ is generated by $t_{i,j}$ and $\ro_j$. Using formulas for 
$\kre{\partial}_1(t_{i,j})$ and $\kre{\partial}_1(\ro_j)$ it is straightforward to check that elements (F1)--(F4) are elements
of $\ker\kre{\partial}_1$. Moreover,
\[2t_{j,1}+2t_{j,2}+\cdots+2t_{j,g}=[t_{a_j}]\otimes 2(\gamma_1+\cdots +\gamma_g)=0,\]
hence $r_{t_j}$ is indeed a relation. Similarly we check that $r_{\ro}$ is a relation.

Observe that using relations $r_{t_j}$ and $r_{\ro}$ we can substitute for $2t_{j,g}$ and $2\ro_1$ respectively, hence
each element in $\gen{\kre{X}}$ can be written as a linear combination
of $t_{i,j},\ro_j$, where each of $t_{1,g},t_{2,g},\ldots,t_{g-1,g},\ro_1$ has the coefficient 0 or 1. Moreover, 
for a given $x\in \gen{\kre{X}}\subset C_1(G)\otimes H_1(N_g;\zz)$ such a combination is unique. Hence for the rest of the proof we assume that
linear combinations of $t_{i,j},\ro_j$ satisfy this condition.

Suppose that $h\in\gen{\kre{X}}\cap \ker\kre{\partial}_1$.
We will show that 
$h$ can be uniquely expressed as a linear combination of generators (F1)--(F4). 

First observe that $h=h_1+h_2$, where
$h_1$ is a combination of generators (F1)--(F2), and $h_2$ does not contain generators of type (F1) nor elements 
$t_{j,j+1}$. Moreover, $h_1$ and $h_2$ are uniquely determined by $h$. 

Next we decompose $h_2=h_3+h_4$, where $h_3$ is a combination of generators (F3) and $h_4$ does not contain 
$\ro_j$ for $j<g$. As before, $h_3$ and $h_4$ are uniquely determined by $h_2$. 

Element $h_4$ has the form 
\[h_4=\sum_{j=1}^{g-1} \alpha_jt_{j,j}+\alpha \ro_g,\]
for some integers $\alpha,\alpha_1,\ldots,\alpha_{g-1}$. Hence
\[  0=\kre{\partial}_1(h_4)=\alpha_1\gamma_1+(\alpha_1+\alpha_2)\gamma_2+\cdots+(\alpha_{g-2}+
\alpha_{g-1})\gamma_{g-1}+(\alpha_{g-1}-2\alpha)\gamma_g. \]
If $g$ is odd this implies that 
\[\alpha_1=\alpha_3=\cdots=\alpha_{g-2}=2k,\ \alpha_2=\alpha_4=\cdots=\alpha_{g-1}=0,\ \alpha=-k,\]
for some $k\in\zz$. For $g$ even we get 
\[\alpha_1=\alpha_3=\cdots=\alpha_{g-1}=2k,\ \alpha=\alpha_2=\alpha_4=\cdots=\alpha_{g-2}=0.\]
In each of these cases $h_4$ is a multiple of the generator (F4).
\end{proof}
By an analogous argument we get
\begin{prop}\label{gen_ker_roz}
Let $g\geqslant 3$ and $G=\cM^h(N_g)$ then $\gen{\kre{X}}\cap\ker\kre{\partial}_1$ is the abelian group 
which admits the presentation with generators: (F1)--(F4), 
\begin{enumerate}
 \item[(F5)] $s_j+s_{j-1}$, where $j$ is even,
\item[(F6)] $s_j-\ro_1-\ro_2-\cdots-\ro_j$, where $j$ is odd. 
\end{enumerate}
The defining relations are $r_{t_j}$, $r_{\ro}$ and
\[r_{s}\!:\begin{cases}
0=2(s_2+s_1)+2(s_4+s_3)+\cdots+2(s_{g-1}+s_{g-2})\\
\qquad\qquad\qquad\qquad\  +2(s_{g}-\ro_1-\ro_2-\cdots-\ro_g)\quad\text{for $g$ odd}\\
            0=2(s_2+s_1)+2(s_4+s_3)+\cdots+2(s_g+s_{g-1})\quad\text{for $g$ even.}
           \end{cases}
\]
\end{prop}
\subsection{Rewriting relations}
Using formula \eqref{eq_rew_rel} we rewrite relations (E1)--(E5) as relations in $H_1(\cM^{h+}(N_g); H_1(N_g;\zz))$.

Relation (E1) is symmetric with respect to $k$ and $j$, hence we can assume that $j+1<k$. This relation gives 
\[\begin{aligned}r^{(E1)}_{k,j:i}\!:0&=([t_{a_k}]+t_{a_k}[t_{a_j}]-[t_{a_j}]-t_{a_j}[t_{a_k}])\otimes \gamma_i\\
&=t_{k,i}+[t_{a_j}]\otimes \psi(t_{a_k}^{-1})\gamma_i-t_{j,i}-[t_{a_k}]\otimes \psi(t_{a_j}^{-1})\gamma_i\\
&=\pm\begin{cases}
   0&\text{if $i\neq k,k+1,j,j+1$}\\
t_{j,k}+t_{j,k+1}&\text{if $i=k$ or $i=k+1$}\\
t_{k,j}+t_{k,j+1}&\text{if $i=j$ or $i=j+1$.}
  \end{cases}
  \end{aligned}
\]
Relation (E2) gives
\[\begin{aligned}r^{(E2)}_{j:i}\!:0=&
([t_{a_j}]+t_{a_{j}}[t_{a_{j+1}}]+t_{a_{j}}t_{a_{j+1}}[t_{a_{j}}]\\
&-[t_{a_{j+1}}]-t_{a_{j+1}}[t_{a_{j}}]-t_{a_{j+1}}t_{a_{j}}[t_{a_{j+1}}])\otimes \gamma_i\\
=&\begin{cases}
                 t_{j,i}-t_{j+1,i}&\text{\hspace{-1.5cm}if $i\neq j,j+1,j+2$}\\
 t_{j,j+2}-t_{j+1,j}&\text{\hspace{-1.5cm}if $i=j+2$}\\
(*)+2(t_{j,j}+t_{j,j+1})&\text{\hspace{-1.5cm}if $i=j$}\\
(*)-(t_{j,j}+t_{j,j+1})-(t_{j+1,j+1}+t_{j+1,j+2})&\text{if $i=j+1$.}
                 \end{cases}
\end{aligned}\]
In the above formula $(*)$ denotes some expression homologous to 0 by previously obtained relations. Carefully 
checking relations $r^{(E1)}_{k,j:i}$ and $r^{(E2)}_{j:i}$ we conclude that generators (F1) generate a cyclic group,
and generators (F2) generate a cyclic group of order at most 2.

We next turn to the relation (E5). It gives
\[\begin{aligned}r^{(E5)}_{j:i}\!:0=&([\ro]+\ro[t_{a_j}]-[t_{a_j}]-t_{a_j}[\ro])\otimes \gamma_i\\
   &=\begin{cases}
        -2t_{j,i}&\text{if $i\neq j,j+1$}\\
 -2t_{j,j}-\ro_j-\ro_{j+1}&\text{if $i=j$}\\
(\ro_j+\ro_{j+1}+2t_{j,j})-2(t_{j,j}+t_{j,j+1})&\text{if $i=j+1$.}
       \end{cases}
\end{aligned}\]
These relations imply that generators (F3) are homologically trivial, and generators (F1) generate at most $\zz_2$.

We now turn to the most difficult relation, namely (E3). This relation gives
\[\begin{aligned}
r^{(E3)}_{i}\!:0&=\sum_{k=0}^{g-1}\sum_{n=1}^{g-1}(t_{a_1}\cdots t_{a_{g-1}})^kt_{a_1}\cdots t_{a_{n-1}}[t_{a_n}]\otimes \gamma_i-\eps\ro_i\\
&=\sum_{n=1}^{g-1}[t_{a_n}]\otimes \psi(t_{a_1}\cdots t_{a_{n-1}})^{-1} \sum_{k=0}^{g-1}\psi(t_{a_1}\cdots t_{a_{g-1}})^{-k}\gamma_i-\eps\ro_i\\
&=\sum_{n=1}^{g-1}[t_{a_n}]\otimes Y_{n} \sum_{k=0}^{g-1}Y_{g}^k\gamma_i-\eps\ro_i.
\end{aligned}
\]
Where $\eps=0$ for $g$ even, $\eps=1$ for $g$ odd, and $Y_n=\psi(t_{a_1}\cdots t_{a_{n-1}})^{-1}$. Using the matrix 
formula for $\psi(t_{a_i}^{-1})$, we obtain
\[Y_n\gamma_i=\begin{cases}
               -\gamma_{i-1}&\text{if $2\leq i\leq n$}\\
\gamma_{i}&\text{if $i> n$}\\
2\gamma_1+\cdots+2\gamma_{n-1}+\gamma_n&\text{if $i=1$.}
              \end{cases}
\]
In particular
\[Y_g^k\gamma_i=(-1)^k\gamma_{i-k},\]
where we subtract indexes modulo $g$. Therefore we have
\[r^{(E3)}_{i}\!:0=\sum_{n=1}^{g-1}[t_{a_n}]\otimes Y_{n}\sum_{k=0}^{g-1}(-1)^k\gamma_{i-k}-\eps\ro_i.\]
In order to simplify computations we replace relations: \[r_1^{(E3)},\ r_2^{(E3)},\ldots,\ r_g^{(E3)}\] with relations:
\[r_1^{(E3)}+r_2^{(E3)},\ r_2^{(E3)}+r_3^{(E3)},\ldots,r_{g-1}^{(E3)}+r_g^{(E3)},\ r_g^{(E3)}.\]
Let us begin with $r_g^{(E3)}$.
\[\begin{aligned}
r_g^{(E3)}\!:0=&\sum_{n=1}^{g-1}[t_{a_n}]\otimes Y_{n}\sum_{k=0}^{g-1}(-1)^k\gamma_{g-k}-\eps\ro_{g}\\
=&\sum_{n=1}^{g-1}[t_{a_n}]\otimes \left(\sum_{k=0}^{g-n-1}(-1)^k\gamma_{g-k}+\sum_{k=g-n}^{g-2}(-1)^{k+1}\gamma_{g-k-1}\right.\\
&\qquad +
 (-1)^{g-1}(2\gamma_1+\cdots+2\gamma_{n-1}+\gamma_n)\bigg)-\eps\ro_{g}.
  \end{aligned}
\]
Since all generators of type (F1) are homologous to a single generator, say $t$, and $2t=0$, the above relation can 
be rewritten as
\[r_g^{(E3)}\!:0=
(g-1)(g-2)t+\sum_{n=1}^{g-1}[t_{a_n}]\otimes \left((-1)^{g-n-1}\gamma_{n+1}+(-1)^{g-1}\gamma_n\right)-\eps\ro_{g}.
\]
If $g$ is even, this gives the relation
\[\begin{aligned}
r_g^{(E3)}\!:0=&(-t_{1,1}+ t_{1,2})+(-t_{2,2}-t_{2,3})+\cdots+(-t_{g-1,g-1}+t_{g-1,g})\\
=&(t_{1,1}+ t_{1,2})-(t_{2,2}+t_{2,3})+\cdots+(t_{g-1,g-1}+t_{g-1,g})\\
&-2(t_{1,1}+t_{3,3}+\cdots+t_{g-1,g-1}).
\end{aligned}\]
If $g$ is odd, we have
\[\begin{aligned}
r_g^{(E3)}\!:0=&(t_{1,1}- t_{1,2})+(t_{2,2}+t_{2,3})+\cdots+(t_{g-1,g-1}+t_{g-1,g})-\ro_g\\
=&-(t_{1,1}+ t_{1,2})+(t_{2,2}+t_{2,3})-\cdots+(t_{g-1,g-1}+t_{g-1,g})\\
&+2(t_{1,1}+t_{3,3}+\cdots+t_{g-2,g-2})-\ro_g.
\end{aligned}\]
In both cases relation $r_g^{(E3)}$ implies that generator (F4) is superfluous.

Now we concentrate on the relation $r_i^{(E3)}+r_{i+1}^{(E3)}$.
\[\begin{aligned}
   r_i^{(E3)}+r_{i+1}^{(E3)}&\!:0=\sum_{n=1}^{g-1}[t_{a_n}]\otimes Y_{n}
\sum_{k=0}^{g-1}(-1)^k\left(\gamma_{i-k}+\gamma_{i+1-k}\right)-\eps(\ro_i+\ro_{i+1})\\
&=\sum_{n=1}^{g-1}[t_{a_n}]\otimes Y_{n}
\left(\gamma_{i+1}+(-1)^{g-1}\gamma_{i+1}\right)-\eps(\ro_i+\ro_{i+1}).
  \end{aligned}
\]
If $g$ is even, this relation is trivial, and if $g$ is odd it gives
\[\begin{aligned}
   r_i^{(E3)}&+r_{i+1}^{(E3)}\!:0=2\sum_{n=1}^{g-1}[t_{a_n}]\otimes Y_{n}
\left(\gamma_{i+1}\right)-(\ro_i+\ro_{i+1})\\
&=2(t_{1,i+1}+\cdots+t_{i,i+1}-t_{i+1,i}-\cdots-t_{g-1,i})-(\ro_i+\ro_{i+1})\\
&=(*)+2(t_{i,i}+t_{i,i+1})-(2t_{i,i}+\ro_i+\ro_{i+1}).
  \end{aligned}
\]
Hence this relation gives no new information.

Relation $(E4)$ gives no new information, hence we proved the following theorem.
\begin{tw}\label{tw:hom:positive}
If $g\geqslant 3$, then 
\[H_1(\cM^{h+}(N_g); H_1(N_g;\zz))=\zz_2\oplus\zz_2.\]
\end{tw}
\subsection{Computing $H_1(\cM^{h}(N_g); H_1(N_g;\zz))$}
If $G=\cM^{h}(N_g)$, then by Proposition \ref{gen_ker_roz} the kernel $\gen{\kre{X}}\cap\ker\kre{\partial}_1$
has two more types of generators: (F5), (F6), and by Theorem \ref{th:004} there are three additional 
relations: (C4),(C5),(C8).
\[\begin{aligned}
   r_i^{(C4)}\!:0&=[s]\otimes \gamma_i+s[s]\otimes\gamma_i=s_i+[s]\otimes\psi(s)\gamma_i\\
&=2(-1)^i\left(s_1+s_2+\cdots+s_{i-1}+\frac{1+(-1)^i}{2}s_i\right).
  \end{aligned}
\]
This (inductively) implies that each generator of type (F5) has order at most 2.
\[\begin{aligned}
   r_i^{(C8)}\!:0&=([\ro]+\ro[s]-[s]-s[\ro])\otimes\gamma_i=\ro_i-2s_i-[\ro]\otimes \psi(s)\gamma_i\\
&=\ro_i-2s_i-(-1)^i(2\ro_1+2\ro_2+\cdots +2\ro_{i-1}+\ro_i)\\
&=\begin{cases}
   -2(s_i-\ro_1-\cdots-\ro_i)&\text{ for $i$ odd}\\
-2(s_{i-1}+s_i)+2(s_{i-1}-\ro_1-\cdots -\ro_{i-1})&\text{ for $i$ even.}
  \end{cases}
  \end{aligned}
\]
This implies that generator (F6) has also order at most 2.
\[\begin{aligned}
   r_i^{(C5)}\!:0&=([t_{a_j}]+t_{a_j}[s]+t_{a_j}s[t_{a_j}]-[s])\otimes \gamma_i\\
&=t_{j,i}+[s]\otimes\psi(t_{a_j}^{-1})\gamma_i+[t_{a_j}]\otimes \psi(s)\psi(t_{a_j}^{-1})\gamma_i-s_i.
  \end{aligned}
\]
If $i\neq j$ and $i\neq j+1$, then 
\[r_i^{(C5)}\!:0=(-1)^i\left(2t_{j,1}+\cdots+2t_{j,i-1}+(1+(-1)^i)t_{j,i}\right),\]
which gives no new information. If $i=j$ or $i=j+1$ and $j$ is odd, then 
\[r_i^{(C5)}\!:0=(*)\pm [(s_j+s_{j+1})+(t_{j,j}+t_{j,j+1})],\]
where as usual $(*)$ denotes homologically trivial element. This relation implies that generators (F5) are superfluous.

Finally, if $i=j$ or $i=j+1$ and $j$ is even then 
\[r_i^{(C5)}\!:0=(*)\pm[(s_{j+1}-\ro_1-\cdots-\ro_{j+1})-(s_{j-1}-\ro_1-\cdots-\ro_{j-1})].\]
This implies that all generators of type (F6) are homologous, hence we proved the following.
\begin{tw}\label{tw:hom:full}
If $g\geqslant 3$, then 
\[H_1(\cM^{h}(N_g); H_1(N_g;\zz))=\zz_2\oplus\zz_2\oplus\zz_2.\]
\end{tw}
\section*{Acknowledgements}
The author wishes to thank the referee for his/her helpful suggestions.
\bibliographystyle{abbrv}

\end{document}